\documentclass[12pt]{amsart}
\usepackage{graphicx}
\usepackage[active]{srcltx}
\usepackage{amsthm,mathrsfs}
\usepackage{a4wide}
\usepackage{amsfonts}
\usepackage{amsmath}
\usepackage{amssymb}
\usepackage{dsfont}
\newtheorem{lemma}{Lemma}
\newtheorem{theorem}{Theorem}

\usepackage{float}
\usepackage{url}

\DeclareMathOperator {\var}{Var_{[0,1]}}

\newcommand {\Z} {\mathbb{Z}}

\newcommand {\R} {\mathbb{R}}
\newcommand {\Q} {\mathbb{Q}}

\newcommand {\ve} {\varepsilon}

\makeatletter\def\blfootnote{\xdef\@thefnmark{}\@footnotetext}\makeatother
\allowdisplaybreaks
\title{\bf Extremal discrepancy behavior of lacunary sequences}

\author{Christoph Aistleitner} 
\address{Department of Mathematics, Graduate School of Science, Kobe University, Kobe 657-8501, Japan}
\email{aistleitner@math.tugraz.at}

\author{Katusi Fukuyama}
\address{Department of Mathematics, Graduate School of Science, Kobe University, Kobe 657-8501, Japan}
\email{fukuyama@math.kobe-u.ac.jp}

\thanks{The first author is supported by a Schr\"odinger scholarship of the Austrian Research
Foundation (FWF). The second author is supported by KAKENHI 24340017 and 24340020}


\begin{document}

\begin{abstract}
In 1975 Walter Philipp proved the law of the iterated logarithm (LIL) for the discrepancy of lacunary sequences: for any sequence $(n_k)_{k \geq 1}$ satisfying the Hadamard gap condition $n_{k+1} / n_k \geq q > 1,~k \geq 1,$ we have
$$
\frac{1}{4 \sqrt{2}} \leq \limsup_{N \to \infty} \frac{N D_N(\{ n_1 x \}, \dots, \{n_N x\})}{\sqrt{2 N \log \log N}} \leq C_q
$$
for almost all $x$. In recent years there has been significant progress concerning the precise value of the limsup in this LIL for special sequences $(n_k)_{k \geq 1}$ having a ``simple'' number-theoretic structure. However, since the publication of Philipp's paper there has been no progress concerning the lower bound in this LIL for generic lacunary sequences $(n_k)_{k \geq 1}$. The purpose of the present paper is to collect known results concerning this problem, to investigate what the optimal value in the lower bound could be, and for which special sequences $(n_k)_{k \geq 1}$ a small value of the limsup in this LIL can be obtained. We formulate three open problems, which could serve as the main targets for future research.
\end{abstract}

\date{}
\maketitle

\section{Introduction and statement of results}

It is a well-known fact that for quickly increasing $(n_k)_{k \geq 1}$ the systems $(\cos 2 \pi n_k x)_{k \geq 1}$ and $(\sin 2 \pi n_k x)_{k \geq 1}$ show properties which are typical for systems of independent, identically distributed (i.i.d.) random variables. For example, if $(n_k)_{k \geq 1}$ is an increasing sequence of positive integers satisfying the Hadamard gap condition
\begin{equation} \label{had}
\frac{n_{k+1}}{n_k} \geq q > 1, ~k \geq 1, 
\end{equation}
then for all $t \in \R$ we have 
$$
\lambda \left\{ x \in (0,1):~ \sqrt{2} \sum_{k=1}^N \cos 2 \pi n_k x < t \sqrt{N} \right\} \to \Phi(t) \qquad \textrm{ as $N \to \infty$,}
$$
and we have
\begin{equation} \label{coslil}
\limsup_{N \to \infty} \frac{\left| \sum_{k=1}^N \cos 2 \pi n_k x \right|}{\sqrt{2 N \log \log N}} = \frac{1}{\sqrt{2}} \qquad \textup{a.e.}
\end{equation}
The first result, where $\lambda$ denotes the Lebesgue measure and $\Phi$ the standard normal distribution function, is a counterpart of the central limit theorem (CLT); it is due to Salem and Zygmund~\cite{sz}. The second result is a counterpart of the law of the iterated logarithm (LIL), and has been proved by 
Erd\H os and G{\'a}l~\cite{egot}. Both results remain valid if the function $\cos 2\pi x$ is replaced by $\sin 2\pi x$. Generally speaking, the almost-independent behavior of $(\cos 2\pi n_k x)_{k \geq 1}$ and $(\sin 2\pi n_k x)_{k \geq 1}$  breaks down unless very strong growth conditions or number-theoretic conditions are imposed on $(n_k)_{k \geq 1}$. Classical survey papers concerning these topics are~\cite{kac,kahane}; more recent ones are~\cite{ab,fuku3}.\\

A sequence $(x_k)_{k \geq 1}$ of real numbers from the unit interval is called \emph{uniformly distributed modulo one} (u.d. mod 1) if for any subinterval $A = [a,b]$ of the unit interval we have
$$
\lim_{N \to \infty} \frac{1}{N} \sum_{k=1}^N \mathds{1}_A (x_k) = b-a.
$$
The quality of uniform distribution is measured in terms of the so-called \emph{discrepancy}. There exist two classical notions of discrepancies, denoted by $D_N$ and $D_N^*$, which are defined by
$$
D_N (x_1, \dots, x_N) = \sup_{0 \leq a \leq b \leq 1} \left| \frac{1}{N} \sum_{k=1}^N \mathds{1}_{[a,b]} (x_k) - (b-a) \right|
$$
and
$$
D_N^* (x_1, \dots, x_N) = \sup_{0 \leq a \leq 1} \left| \frac{1}{N} \sum_{k=1}^N \mathds{1}_{[0,a]} (x_k) - a \right|,
$$
for points $x_1, \dots, x_N$ in $[0,1]$. $D_N$ is called the \emph{discrepancy}, and $D_N^*$ is called the \emph{star-discrepancy}. It is easily seen that for any points we have $D_N^* \leq  D_N \leq 2D_N^*$, and that an infinite sequence $(x_k)_{k \geq 1}$ is u.d. mod 1 if and only if its discrepancy tends to zero as $N \to \infty$. Koksma's inequality (see e.g.~\cite[Chapter 2, Theorem 5.1]{knu}) states that
\begin{equation} \label{koks}
\left| \int_0^1 f(x)~dx - \frac{1}{N} \sum_{k=1}^N f(x_k) \right| \leq \left(\var f\right) ~D_N^*(x_1, \dots, x_N)
\end{equation}
for any function $f$ of bounded variation on $[0,1]$; the multidimensional generalization of this inequality is the reason why discrepancy theory plays an important role in multidimensional numerical integration, see e.g.~\cite{dipi,dts,knu}.\\

In his seminal paper of 1916, Weyl~\cite{weyl} showed that for any sequence $(n_k)_{k \geq 1}$ of distinct positive integers the sequence $(\{n_k x\})_{k \geq 1}$, where $\{ \cdot \}$ denotes the fractional part function, is u.d. mod 1 for almost all $x$. Determining the precise asymptotic order of the discrepancy of $(\{n_k x\})_{k \geq 1}$ for typical $x$ is a very difficult problem, which is only solved for very few sequences $(n_k)_{k \geq 1}$ (for example when $n_k=k,~k \geq 1$; see \cite{schoi}). However, due to the analogy between lacunary series and sequences of independent random variables described in the first paragraph, very precise metric results for the asymptotic order of the discrepancy of $(\{n_k x\})_{k \geq 1}$ can be obtained if $(n_k)_{k \geq 1}$ is quickly increasing, as we will show below.\\

One could expect that the almost-independence property of lacunary series remains valid if the functions $\cos 2\pi x$ and $\sin 2\pi x$ are replaced by an other ``nice'' 1-periodic function $f$. However, this is only the case in a significantly weakened form, even if $f$ is a trigonometric polynomial. This is most easily seen by considering
$$
f(x) = \cos 2 \pi x - \cos 4 \pi x, \qquad n_k = 2^k, \quad k \geq 1.
$$
In this case the sum $\sum_{k=1}^N f(n_k x)$ is a telescoping sum, and it is obvious that the normalized partial sums will satisfy neither the CLT nor the LIL. Another example, independently due to Erd\H os and Fortet (see~\cite{kac}) is the following: set
$$
f(x) = \cos 2 \pi x + \cos 4 \pi x, \qquad n_k = 2^k-1, \quad k \geq 1.
$$
In this case the CLT fails, while the LIL holds in the modified form
\begin{equation} \label{erdfort}
\limsup_{N \to \infty} \frac{\left| \sum_{k=1}^N f(n_k x) \right|}{\sqrt{2 N \log \log N}} = \sqrt{2} |\cos \pi x| \qquad \textup{a.e.}
\end{equation}
Thus the precise form of the LIL may fail for Hadamard lacunary $(n_k)_{k \geq 1}$ in the case of general $f$. However, in this case by a result of Takahashi~\cite{taka} we still have the following upper-bound version of the LIL: for $(n_k)_{k \geq 1}$ satisfying~\eqref{had} and $f$ satisfying
\begin{equation} \label{f}
f(x+1)=f(x), \qquad \int_0^1 f(x) ~dx = 0, \qquad \var f < \infty,
\end{equation}
we have
\begin{equation} \label{taka}
\limsup_{N \to \infty} \frac{\left| \sum_{k=1}^N f(n_k x) \right|}{\sqrt{2 N \log \log N}} \leq C \qquad \textup{a.e.},
\end{equation}
for some appropriate constant $C$ (depending on $f$ and the growth factor $q$). From a probabilistic point of view, the star-discrepancy and extremal discrepancy are a version of the (one-sided and two-sided, respectively) Kolmogorov-Smirnov statistic. The Chung-Smirnov law of the iterated logarithm implies that for $X_1, X_2,\dots$ being i.i.d. random variables having uniform distribution on $[0,1]$ we have
$$
\limsup_{N \to \infty} \frac{N D_N(X_1, \dots, X_N)}{\sqrt{2 N \log \log N}} = \frac{1}{2} \qquad \textup{a.s.},
$$
and the same result holds if $D_N$ is replaced by $D_N^*$. Erd\H os and G\'al conjectured that an upper-bound version of the Chung--Smirnov LIL should also hold if the sequence of i.i.d. random variables is replaced by the ``almost independent'' system $(\{ n_k x\})_{k \geq 1}$. This was confirmed by Philipp~\cite{plt} in 1975; he showed that for $(n_k)_{k \geq 1}$ satisfying~\eqref{had} we have
\begin{equation} \label{phil}
\limsup_{N \to \infty} \frac{N D_N(\{n_1 x\}, \dots, \{n_N x\})}{\sqrt{2 N \log \log N}} \leq C_q \qquad \textup{a.e.},
\end{equation}
where $C_q$ depends on $q$. Together with Koksma's inequality,~\eqref{phil} implies~\eqref{taka}. On the other hand, by~\eqref{coslil}, Koksma's inequality and the fact that $\var \cos 2 \pi x = 4$ we have
\begin{equation} \label{lowerb}
\limsup_{N \to \infty} \frac{N D_N^*(\{n_1 x\}, \dots, \{n_N x\})}{\sqrt{2 N \log \log N}} \geq \frac{1}{4 \sqrt{2}} \qquad \textup{a.e.}
\end{equation}

In the sequel, for a given sequence $(n_k)_{k \geq 1}$ we set 
$$
\Lambda(x) = \limsup_{N \to \infty} \frac{N D_N(\{n_1 x\}, \dots, \{n_N x\})}{\sqrt{2 N \log \log N}}, \qquad \Lambda^*(x) = \limsup_{N \to \infty} \frac{N D_N^*(\{n_1 x\}, \dots, \{n_N x\})}{\sqrt{2 N \log \log N}}.
$$
In 2008 Fukuyama~\cite{fuku1} calculated the precise value of the limsup in the LIL for the discrepancy of $(\{n_k x\})_{k \geq 1}$ for special sequences of the form $n_k = \theta^k,~ k \geq 1$. Amongst other results, for such sequences he obtained the following:
\begin{equation} \label{lambdafuku}
\Lambda = \Lambda^* = \left\{ \begin{array}{ll} \frac{\sqrt{42}}{9} & \textrm{if $\theta=2$,}\\ \frac{\sqrt{(\theta+1)\theta(\theta-2)}}{2 \sqrt{(\theta-1)^3}} & \textrm{if $\theta \geq 4$ is even,} \\ \frac{\sqrt{\theta+1}}{2 \sqrt{\theta-1}} & \textrm{if $\theta \geq 3$ is odd,} \\ \frac{1}{2} & \textrm{if $\theta^r \not\in \Q$ for all $r \in \Q$}, \end{array} \right.
\end{equation}
for almost all $x$. Such results indicate that there is a close connection between number-theoretic properties of $(n_k)_{k \geq 1}$ and the precise asymptotic order of $\sum f(n_k x)$ and of the discrepancy of $(\{n_k x\})_{k \geq 1}$. It turns out that the number of solutions of certain Diophantine equations plays an important role. For $j_1 \geq 1,j_2 \geq 1,\nu \in \Z$ and $N \geq 1$ we set
\begin{equation} \label{dioph}
S(j_1,j_2,\nu,N) = \# \Big\{(k,l):~(j_1,k) \neq (j_2,l), ~1 \leq k,l \leq N,~j_1 n_k - j_2 n_l = \nu \Big\}.
\end{equation}
For some $d \geq 1$ we say that $(n_k)_{k \geq 1}$ satisfies condition $\mathbf{D}_d$ if for all $1 \leq j_1, j_2 \leq d$ there exist real numbers $\gamma_{j_1,j_2,\nu}$ such that
\begin{equation} \label{dioph2}
\left| \frac{S(j_1,j_2,\nu,N)}{N} - \gamma_{j_1,j_2,\nu} \right| = \mathcal{O} \left( \frac{1}{(\log N)^{1 + \ve}} \right)
\end{equation}
for some fixed $\ve>0$, uniformly in $\nu \in \Z$. Furthermore we say that $(n_k)_{k \geq 1}$ satisfies condition $\mathbf{D}$ if it satisfies $\mathbf{D}_d$ for all $d \geq 1$. Assume that $(n_k)_{k \geq 1}$ satisfies~\eqref{had} and condition $\mathbf{D}$, and that $f$ is a function satisfying~\eqref{f}. Then, writing
$$
f(x) = \sum_{j=1}^\infty \left(a_j \cos 2 \pi j x + b_j \sin 2 \pi j x \right)
$$
and setting
\begin{equation} \label{sigma}
\sigma_f^2 (x) = \|f\|^2 + \sum_{\nu = -\infty}^\infty \sum_{j_1,j_2=1}^\infty \frac{\gamma_{j_1,j_2,\nu}}{2} \Big( \left(a_{j_1} a_{j_2} + b_{j_1} b_{j_2} \right) \cos 2 \pi \nu x + \left(b_{j_1}a_{j_2} - a_{j_1} b_{j_2}\right) \sin 2 \pi \nu x \Big),
\end{equation}
where $\| \cdot \|$ denotes the $L^2(0,1)$ norm, we have
\begin{equation} \label{generalLIL}
\limsup_{N \to \infty} \frac{\left| \sum_{k=1}^N f(n_k x)\right|}{\sqrt{2 N \log \log N}} = \sigma_f(x) \qquad \textup{a.e.}
\end{equation}
(see~\cite[Theorem 2]{aist1}). Writing $\sigma_{[a,b]}$ for the function defined in~\eqref{sigma} for the (centered, extended with period 1) indicator functions
$$
\mathbf{I}_{[a,b]}(x) = \sum_{m \in \Z} \mathds{1}_{[a,b]} (x + m) - (b-a), \qquad \textrm{where $0 \leq b-a < 1$},
$$
then for the discrepancy we have (see~\cite[Theorem 3]{aist1})
\begin{equation} \label{discrv}
\Lambda^*(x) = \sup_{0 \leq a \leq 1} \sigma_{[0,a]}(x) \qquad \textup{a.e.} \quad \textrm{and} \qquad \Lambda(x) = \sup_{0 \leq a \leq b \leq 1} \sigma_{[a,b]}(x) \quad \textup{a.e.}
\end{equation}

In particular if the Diophantine equations in~\eqref{dioph} have a ``small'' number of solutions, namely if $\gamma_{j_1,j_2,\nu}=0$ for all $j_1,j_2,\nu$, then $\Lambda^* = \Lambda = \frac{1}{2}$ for a.e. $x$; that is, in this case we have the same limsup as in the Chung--Smirnov LIL for i.i.d. random variables (see~\cite{aist4}). On the other hand, if the number of solutions of these Diophantine equations is large, then $\Lambda$ and $\Lambda^*$ are in general different from $1/2$, and can even show the ``irregular'' behavior of not being equal to a constant for almost all $x$ (see~\cite{aist2,aist3,fuku2}).\\

Despite considerable efforts to construct lacunary sequences $(n_k)_{k \geq 1}$ with extremal discrepancy behavior, so far no lacunary $(n_k)_{k \geq 1}$ has been found for which
\begin{itemize} 
\item \quad $\Lambda^* < \frac{1}{2}$ on a set of positive measure, or
\item \quad $\|\Lambda^*\|< \frac{1}{2}$, or
\item \quad $\Lambda < \frac{1}{2}$ on a set of positive measure.
\end{itemize}
The purpose of the present paper is to construct a sequence $(n_k)_{k \geq 1}$ for which the first of these three properties holds, and to demonstrate why conventional constructions cannot provide an example of a sequence for which either the second or third property holds. We will also show that the lower bound in~\eqref{lowerb} can be improved if $D_N^*$ is replaced by $D_N$.

\begin{theorem} \label{th1}
Let $(n_k)_{k \geq 1}$ be defined by 
$$
n_k = \left\{ \begin{array}{ll} 3^{k^2} & \textrm{if $k$ is odd} \\ 3^{(k-1)^2+1}-1 & \textrm{if $k$ is even.} \end{array} \right.
$$
Then for almost all $x \in [0,1]$ we have
$$
\Lambda^*(x) = \left\{ \begin{array}{ll} \sqrt{\frac{-3x^2-x+2}{6}} & \textrm{if $0 \leq x \leq \frac{1}{6},$} \\ 
\sqrt{\frac{-24x+25}{72}} & \textrm{if $\frac{1}{6} \leq x \leq \frac{3}{8},$} \\ \sqrt{\frac{2}{9}} & \textrm{if $\frac{3}{8} \leq x \leq \frac{1}{2},$} \\ \Lambda^*(1-x) & \textrm{if $\frac{1}{2} < x \leq 1$.} \end{array} \right.
$$
In particular we have $\Lambda^*(x) < 1/2$ for almost all $x \in (7/24,17/24)$.
\end{theorem}

\vspace{.3cm}
\begin{figure}[H]
\begin{center}
\includegraphics[angle=0,width=80mm]{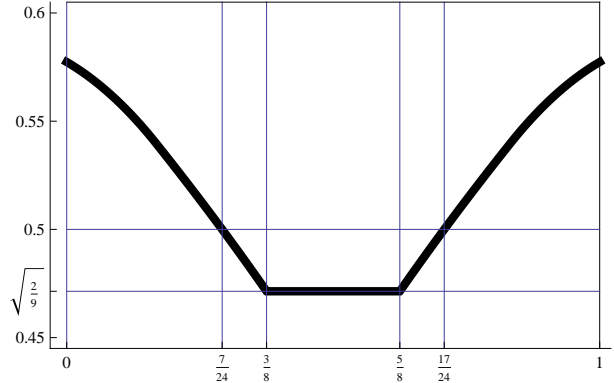}\\
\caption{The function $\Lambda^*(x)$ for the sequence $(n_k)_{k \geq 1}$ defined in Theorem~\ref{th1}.}
\end{center}
\end{figure}
The construction of the sequence in Theorem~\ref{th1} is very similar to the constructions in~\cite{aist2,aist3}, and makes use of~\eqref{generalLIL} and the simple structure of the sequence $(n_k)_{k \geq 1}$ with respect to linear Diophantine equations as those in~\eqref{dioph}. Sequences constructed in this way are (relatively simple) examples of sequences satisfying condition $\mathbf{D}$; and sequences satisfying condition $\mathbf{D}$ are essentially the only sequences for which the values of $\Lambda$ and $\Lambda^*$ can be calculated with currently available tools. The next theorem shows, roughly speaking, that the other two open problems on the previous list cannot be solved using such a standard construction. If a sequence satisfying one of the last two of the three points above exists, then it cannot satisfy condition $\mathbf{D}$ and accordingly must have a somewhat irregular Diophantine structure.
 
\begin{theorem} \label{th2}
If $(n_k)_{k \geq 1}$ satisfies the Diophantine condition $\mathbf{D}$, then
$$
\| \Lambda^* \| \geq \frac{1}{2}, \qquad \textrm{and} \qquad \Lambda \geq \frac{1}{2} \qquad \textup{a.e.}
$$
\end{theorem}

The next theorem is an improved version of the general lower bound~\eqref{lowerb}, which follows from a version of Koksma's inequality for symmetric functions (see Lemma~\ref{modkoksma} below).

\begin{theorem} \label{th3}
For any sequence of positive integers $(n_k)_{k \geq 1}$ satisfying~\eqref{had} we have
$$
\Lambda \geq \frac{1}{2 \sqrt{2}} \qquad \textup{a.e.}
$$
\end{theorem}

Some important problems concerning lower bounds for $\Lambda$ and $\Lambda^*$ remain open. In all three questions below it is assumed that $(n_k)_{k \geq 1}$ satisfies the Hadamard gap condition~\eqref{had}.

\begin{itemize}
\item \emph{Open problem 1: } Is it possible that $\| \Lambda^* \| < \frac{1}{2}$?
\item \emph{Open problem 2: } Is it possible that $\Lambda^* < 1/2$ almost everywhere?
\item \emph{Open problem 3: } Is it possible that $\Lambda < \frac{1}{2}$ on a set of positive measure?
\end{itemize}

Theorem~\ref{th2} suggests that the answer of all three problems could be negative, but this is by no means certain. Problem 1 and Problem 2 are related: if the answer of Problem 1 is ``no'', then the answer of Problem 2 must also be ``no''.\\

The second part of Theorem~\ref{th2} is proved with an argument which involves taking the average $L^2$ norm over all intervals of fixed length. This argument is related to an interesting phenomenon, which apparently has not been observed before. Since this observation might be fruitful for further investigations, we state it as a theorem below. Note that for this theorem no growth conditions on the numbers $n_k$ are necessary.

\begin{theorem} \label{th4}
For any $N \geq 1$ let $n_1, \dots, n_N$ be distinct positive integers. Let $z \in (0,1)$ be fixed. Then for any $x \in [0,1]$ we have
$$
\int_0^1 \int_0^1 \left( \sum_{k=1}^N \mathbf{I}_{[a,a+z]} (n_k x) \right)^2 ~dx ~da= z (1-z) N.
$$
\end{theorem}

Theorem~\ref{th4} should be compared with the fact that for a sequence $X_1, \dots, X_N$ of i.i.d. $[0,1]$-uniformly distributed random variables for any $a,z$ we have
$$
\int_\Omega \left( \sum_{k=1}^N \mathbf{I}_{[a,a+z]} (X_k(\omega)) \right)^2 ~d\omega = z (1-z) N.
$$
Such a result does \emph{not} hold for the (dependent) random variables $\{n_1 x\}, \dots, \{n_N x\}$ for all individual values of $a$; however, as Theorem~\ref{th4} shows, it holds ``on average'' by integrating over all indicator functions of equal length.

\section{Auxiliary results}

In this section we first state several auxiliary results, and give the proofs for them afterward.

\begin{lemma} \label{lemma1}
For the sequence defined in Theorem~\ref{th1} we have for any fixed $1 \leq j_2 \leq j_1$ that 
$$
S(j_1,j_2,\nu,N) = \left\{ \begin{array}{ll} \frac{N}{2} + \mathcal{O} (1) & \textrm{if $j_1 = 3 j_2$ and $j_2 = \nu$},\\
\mathcal{O} (1) & \textrm{otherwise}, \end{array} \right.
$$
uniformly in $\nu \in \Z$.
\end{lemma}

\begin{lemma} \label{lemmagamma}
Let $f$ be a function satisfying~\eqref{f}, and write
$$
f(x) = \sum_{j=1}^\infty \left(a_j \cos 2 \pi j x + b_j \sin 2 \pi j x \right).
$$
Let $(n_k)_{k \geq 1}$ be a sequence satisfying~\eqref{had} and the Diophantine condition $\mathbf{D}$. Then we have 
\begin{equation} \label{lemmalhs}
\sum_{\nu = -\infty}^\infty \sum_{j_1,j_2=1}^\infty \frac{\gamma_{j_1,j_2,\nu}}{2} \left( |a_{j_1} a_{j_2}| + |b_{j_1} b_{j_2}| + |b_{j_1}a_{j_2}| + |a_{j_1} b_{j_2}| \right) \leq \frac{2 (\var f)^2 q}{3 (q-1)},
\end{equation}
where $q$ is the growth factor in~\eqref{had}.
\end{lemma}

Lemma~\ref{lemmagamma} provides a uniform upper bound for the function $\sigma_f(x)$ defined in~\eqref{sigma}, and will be needed for the proof of the second part of Theorem~\ref{th2}.

\begin{lemma} \label{modkoksma}
Let $f(x)$ be a function having bounded variation $\var f$ on $[0,1]$, which is symmetric around 1/2; that is, it satisfies $f(1/2-y)=f(1/2+y),~y \in [0,1/2]$. Then for any points $x_1, \dots, x_N \in [0,1]$ we have 
$$
\left| \int_0^1 f(x) ~dx - \frac{1}{N} \sum_{k=1}^N f(x_k) \right| \leq \frac{\var f}{2} D_N(x_1, \dots, x_N).
$$
\end{lemma}
This lemma is a version of Koksma's inequality~\eqref{koks}, adapted to the case of symmetric functions. Note the two differences between this Lemma and the original version of Koksma's inequality: the error estimate $\var f$ is replaced by $\left(\var f \right)/2$, and the star-discrepancy $D_N^*$ is replaced by $D_N$.

\begin{proof}[Proof of Lemma~\ref{lemma1}]
The construction of the sequence $(n_k)_{k \geq 1}$ in Theorem~\ref{th1} is exactly the same as the construction of the sequence in~\cite{aist2}, except that the base $2$ has been replaced by $3$. Thus Lemma~\ref{lemma1} is a variant of~\cite[Lemma 2]{aist2}, and can be shown in the same way.
\end{proof}

\begin{proof}[Proof of Lemma~\ref{lemmagamma}]
Let $K = \var f$. Then for the Fourier coefficients of $f$ we have (see for example~\cite[Theorem 1]{izu} or \cite[p. 48]{zyg}) that
$$
|a_j| \leq \frac{K}{\pi j}, \qquad |b_j| \leq \frac{K}{\pi j}, \qquad j \geq 1. 
$$
Thus the left-hand side of~\eqref{lemmalhs} is bounded by
\begin{equation}\label{rhsrew}
\sum_{\nu = -\infty}^\infty \sum_{j_1,j_2=1}^\infty \frac{\gamma_{j_1,j_2,\nu}}{2} \frac{4 K^2}{\pi^2 j_1 j_2} = \frac{4}{\pi^2} \sum_{\nu = -\infty}^\infty \sum_{1 \leq j_1 < j_2 < \infty} \frac{\gamma_{j_1,j_2,\nu} K^2}{j_1 j_2},
\end{equation}
where the equality follows from the fact that for all $j_1,j_2,\nu$ we have $\gamma_{j_1,j_2,\nu}=\gamma_{j_2,j_1,-\nu}$, and that by~\eqref{had} we necessarily have $\gamma_{j_1,j_2,\nu}=0$ whenever $j_1=j_2$. We can rewrite the right-hand side of~\eqref{rhsrew} in the form
\begin{equation} \label{rhsest2}
\frac{4}{\pi^2} \sum_{j_1=1}^\infty \sum_{r=0}^\infty ~\sum_{j_1 q^r < j_2 \leq j_1 q^{r+1}} ~\sum_{\nu = -\infty}^\infty \frac{\gamma_{j_1,j_2,\nu} K^2}{j_1 j_2},
\end{equation}
where $q$ is the growth factor from~\eqref{had}. From the definition of the numbers $\gamma_{j_1,j_2,\nu}$ for any fixed $j_1$ and $r$ we necessarily have
\begin{equation} \label{sum1}
\sum_{j_1 q^r < j_2 \leq j_1 q^{r+1}} ~\sum_{\nu = -\infty}^\infty \gamma_{j_1,j_2,\nu} \leq 1.
\end{equation}
In fact, assume that there exist $j_1$ and $r$ such that~\eqref{sum1} does not hold; then there must exist \emph{finitely} many triplets $(j_1,j_2^{(i)},\nu^{(i)})$ with $j_1 q^r < j_2^{(i)} \leq j_1 q^{r+1}$ and $\nu^{(i)} \in \Z$ such that
\begin{equation} \label{sum2}
\sum_i \gamma_{j_1,j_2^{(i)},\nu^{(i)}} > 1.
\end{equation}
By assumption we have $n_{k_1} / n_{k_2} \not\in (q^{-1},q)$, for $k_1 \neq k_2$. This means that for sufficiently large $k_1$ it is not possible that there exist numbers $j_1,j_2,j_3$ satisfying $j_1 q^r < j_2,j_3 \leq j_1 q^{r+1}$ together with different indices $k_2,k_3$ and numbers $\nu_1,\nu_2$ from the \emph{finite} set $\bigcup_i \nu^{(i)}$ such that both equalities 
$$
j_1 n_{k_1} - j_2 n_{k_2} = \nu_1 \qquad \textrm{and} \qquad j_1 n_{k_1} - j_3 n_{k_3} = \nu_2
$$
hold. In view of the definition of the numbers $\gamma$ this implies that actually 
$$
\sum_i \gamma_{j_1,j_2^{(i)},\nu^{(i)}} \leq 1,
$$
which is in contradiction with~\eqref{sum2}. Consequently,~\eqref{sum1} must be true, and by~\eqref{rhsest2} and~\eqref{sum1} we get the upper bound
$$
\frac{4}{\pi^2} \sum_{j_1=1}^\infty \sum_{i=0}^\infty \frac{K^2}{j_1^2 q^{i}} = \frac{2 K^2q}{3 (q-1)},
$$
which by our previous consideration is also an upper bound for the left-hand side of~\eqref{lemmalhs}. This proves the Lemma.
\end{proof}

\begin{proof}[Proof of Lemma~\ref{modkoksma}]
Let $f$ and $x_1, \dots, x_N \in [0,1]$ be given. We define points $\bar{x}_1, \dots, \bar{x}_N$ by
$$
\bar{x}_k = \left\{ \begin{array}{ll} x_k & \textrm{if $x_k \leq 1/2$} \\ 1-x_k & \textrm{if $x_k > 1/2$,} \end{array} \right.
$$
for $1 \leq k \leq N$. Then clearly all points $\bar{x}_1, \dots, \bar{x}_N$ lie in the interval $[0,1/2]$, and by the symmetry of $f$ we have $f(x_k)=f(\bar{x}_k),~1 \leq k \leq N$. Also by the symmetry of $f$ we have
\begin{eqnarray}
& & \left| \int_0^1 f(x) ~dx - \frac{1}{N} \sum_{k=1}^N f(x_k) \right| \nonumber\\
& = & \left| \int_0^{1} f(x/2) ~dx - \frac{1}{N} \sum_{k=1}^N f(2\bar{x}_k/2) \right| \nonumber\\
& \leq & \var f(x/2)~ D_N^*(2 \bar{x}_1, \dots, 2 \bar{x}_N), \label{koksequ}
\end{eqnarray}
where the last line follows from Koksma's inequality for the function $f(x/2)$. Note that by the symmetry of $f$ we have 
\begin{equation} \label{var2}
\var f(x/2) = \left(\var f(x)\right)/2.
\end{equation}
For any $a \in [0,1]$ and any $k$ we have
$$
2 \bar{x}_k \in [1-a,1] \qquad \textrm{if and only if} \qquad x_k \in [1/2-a/2,1/2+a/2].
$$
Thus we have
\begin{eqnarray*}
D_N^*(2 \bar{x}_1, \dots, 2 \bar{x}_N) & = & \sup_{0 \leq a \leq 1} \left| \frac{1}{N} \sum_{k=1}^N \mathds{1}_{[1-a,1]} (2 \bar{x}_k) - a \right| \\
& = & \sup_{0 \leq a \leq 1} \left| \frac{1}{N} \sum_{k=1}^N \mathds{1}_{[1/2-a/2,1/2+a/2]} (x_k) -a \right| \\
& \leq & D_N (x_1, \dots, x_N).
\end{eqnarray*}
Together with~\eqref{koksequ} and~\eqref{var2} this proves the lemma.
\end{proof}

\section{Proofs}

\begin{proof}[Proof of Theorem~\ref{th1}:~]
Let $[0,a] \subset [0,1]$ be given. Then the Fourier series of the centered indicator function $\mathbf{I}_{[0,a]}(x)$ is given by
$$
\mathbf{I}_{[0,a]} (x) \sim \sum_{j=1}^\infty  \underbrace{\frac{\sin 2 \pi j a}{\pi j}}_{=:~a_j ~= ~a_j(a)} \cos 2 \pi j x  + \underbrace{\left( \frac{1-\cos 2 \pi j a}{\pi j} \right)}_{=:~b_j ~= ~b_j (a)} \sin {2 \pi j x}.
$$
By Lemma~\ref{lemma1} we have, for the sequence $(n_k)_{k \geq 1}$ defined in Theorem~\ref{th1}, that for $1 \leq j_2 \leq j_1$
$$
\gamma_{j_1,j_2,\nu} = \left\{ \begin{array}{ll} \frac{1}{2} & \textrm{if $j_1 = 3 j_2$ and $j_2 = \nu$},\\
0 & \textrm{otherwise}, \end{array} \right.
$$
Note that $S(j_1,j_2,\nu,N) = S(j_2,j_1,-\nu,N)$. Using this relation Lemma~\ref{lemma1} can also be used to calculate $\gamma_{j_1,j_2,\nu}$ in the case $j_2 \geq j_1$, and we have 
$$
\gamma_{j_1,j_2,\nu}=\gamma_{j_2,j_1,-\nu}.
$$

Thus according to formula~\eqref{sigma}, for the function $\sigma^2_{[0,a]}(x)$ for $\mathbf{I}_{[0,a]}(x)$ we have
\begin{eqnarray}
\sigma^2_{[0,a]} (x) & = & a (1-a) + \sum_{j=1}^\infty \frac{1}{4} \Big( \left(a_{3j} a_{j} + b_{3 j} b_{j} \right) \cos 2 \pi j x + \left(b_{3j}a_{j} - a_{3j} b_{j}\right) \sin 2 \pi j x \Big) \nonumber\\
& & + \sum_{j=1}^\infty \frac{1}{4} \Big( \left(a_{j} a_{3j} + b_{j} b_{3j} \right) \cos 2 \pi (-j) x + \left(b_{j}a_{3j} - a_{j} b_{3j}\right) \sin 2 \pi (-j) x \Big) \nonumber\\
& = & a (1-a)  + \sum_{j=1}^\infty \frac{1}{2} \Big( \left(a_{3j} a_{j} + b_{3 j} b_{j} \right) \cos 2 \pi j x + \left(b_{3j}a_{j} - a_{3j} b_{j}\right) \sin 2 \pi j x \Big).  \label{sigmaex}
\end{eqnarray}
For the Fourier coefficients $a_j,b_j$ we have the relation
\begin{equation} \label{fourrel}
a_j(1-a) = -a_j(a) \qquad \textrm{and} \qquad b_j(1-a) = b_j(a).
\end{equation}
The convolution theorem of Fourier analysis states in its real form that for two functions $g,h$ given by
$$
g(x) \sim \sum_{j=1}^\infty c_j \cos 2 \pi j x+ d_j \sin 2 \pi j x, \qquad h(x) \sim \sum_{j=1}^\infty \bar{c}_j \cos 2 \pi j x+ \bar{d}_j \sin 2 \pi j x,
$$
the function
$$
\frac{1}{2} \sum_{j=1}^\infty \left(c_j \bar{c}_j - d_j \bar{d}_j\right) \cos 2 \pi j x + \left(c_j \bar{d}_j + \bar{c}_j d_j\right) \sin 2 \pi j x
$$
is the Fourier series of
$$
\int_0^1 f(t) g(x-t) ~dt.
$$
Let $g(x)= -\mathbf{I}_{[0,1-a]}$ and $h(x)=\mathbf{I}_{[0,\{3 a\}]}$. Then by~\eqref{fourrel} the Fourier coefficients $c_j,d_j$ of $g$ satisfy
\begin{equation} \label{fou1}
c_j = a_j, \qquad d_j = - b_j, \qquad j \geq 1,
\end{equation}
and the Fourier coefficients $\bar{c}_j,\bar{d}_j$ of $h$ satisfy
\begin{equation} \label{fou2}
\bar{c}_j = 3 a_{3j}, \qquad \bar{d}_j = 3 b_{3j}, \qquad j \geq 1.
\end{equation}
Using the convolution theorem for the functions $g$ and $h$ defined in this way, and comparing the Fourier coefficients in~\eqref{fou1},~\eqref{fou2} with those in equation~\eqref{sigmaex}, we get the expression
\begin{eqnarray*}
\sigma^2_{[0,a]}(x) & = & a(1-a) - \frac{1}{3} \int_0^1 \mathbf{I}_{[0,1-a]} (t) ~\mathbf{I}_{[0,\{ 3 a \}]}(x-t)~dt \\
& =  & a(1-a) + \frac{(1-a)\{3a\}}{3} - \frac{1}{3} \int_0^1 \mathds{1}_{[0,1-a]} (t) ~\mathds{1}_{[0,\{ 3 a \}]}(\{x-t\})~dt.
\end{eqnarray*}
for the function $\sigma^2_{[0,a]}(x)$. Using this formula for $a \in [2/3,1]$ we see that for all $x$ we have
$$
\sigma^2_{[0,a]}(x) \leq a(1-a) + \frac{(1-a)\{3a\}}{3} = a(1-a) + \frac{(1-a)(3a-2)}{3} = - 2a^2+\frac{8a}{3}-\frac{2}{3}.
$$
It is easily seen that
$$
\max_{a \in [2/3,1]} \left( - 2a^2+\frac{8a}{3}-\frac{2}{3} \right) = \frac{2}{9},
$$
and thus
\begin{equation} \label{sigma1}
\max_{a \in [2/3,1]} \sigma_{[0,a]}^2(x) \leq \frac{2}{9} \qquad \textrm{for all $x \in [0,1]$}.
\end{equation}
Furthermore, if $a \in [0,1/3]$, then the length of $[0,1-a]$ plus the length of $[0,\{3a\}]=[0,3a]$ is $1+2a$; consequently, even if the second interval is translated (mod 1), there must be an overlap of length at least $2a$. Thus we have
\begin{equation} \label{sigma2}
\sigma^2_{[0,a]} (x) \leq a(1-a) + \frac{(1-a)3a}{3} - \frac{2a}{3} \leq \frac{2}{9} \qquad \textrm{for all $a \in [0,1/3]$ and $x \in [0,1]$},
\end{equation}
where the last inequality again follows by standard methods. The only complicated case is when $a \in [1/3,2/3]$. In this case we have
\begin{equation} \label{integ}
\sigma^2_{[0,a]} (x) = a(1-a) + \frac{(1-a)(3a-1)}{3} - \frac{1}{3} \underbrace{\int_0^1 \mathds{1}_{[0,1-a]} (t) ~\mathds{1}_{[0,3a-1]}(\{x-t\})~dt}_{=: I_1}.
\end{equation}
If we assume that $0 \leq x \leq 1-a$, then we have the following values for the integral $I_1$ in~\eqref{integ}:\\

\begin{tabular}{|l|l|}
\hline
Additional assumption & Value of $I_1$ \\
\hline
If $3a - 1 \leq x$ & $3a-1$ \\
\hline 
If $x \leq 3a-1 \leq x + a$ & $x$ \\
\hline
If $x + a \leq 3a-1$ & $2a-1$ \\
\hline
\end{tabular}\\
Thus we can explicitly describe the values of $\sigma_{[0,a]}(x)$ for all $a \in [1/3,2/3]$ and $x \in [0,1]$ satisfying $0 \leq x \leq 1-a$. We note that the function $\sigma_{[0,a]}(x)$ satisfies the relation $\sigma_{[0,a]}(x)=\sigma_{[0,1-a]}(1-x)$; this can be seen for example from~\eqref{sigmaex} and~\eqref{fourrel}. Consequently the values of $\sigma_{[0,a]}(x)$ under the additional assumption $1-a \leq x \leq 1$ can be obtained by transforming the corresponding cases of $0 \leq x \leq 1-a$. Overall, we have a full explicit description of $\sigma_{[0,a]}(x)$ for all values $a \in [1/3,2/3]$ and $x \in [0,1]$. Consequently, to calculate $\max_{a \in [1/3,2/3]} \sigma^2_{[0,a]}(x) $ it remains to solve a (simple, but quite laborious) maximization problem with several different cases. Using standard methods, it can be shown that for given $x \in [0,1]$ the largest possible value of $\sigma^2_{[0,a]}(x)$ (for some appropriate $a$, depending on $x$) is 
$$
\max_{a \in [1/3,2/3]} \sigma^2_{[0,a]}(x) = \left\{ \begin{array}{ll} \frac{-3x^2-x+2}{6} & \textrm{if $0 \leq x \leq \frac{1}{6},$} \\ 
\frac{-24x+25}{72} & \textrm{if $\frac{1}{6} \leq x \leq \frac{3}{8},$} \\ \frac{2}{9} & \textrm{if $\frac{3}{8} \leq x \leq \frac{1}{2},$} \\ \max_{a \in [1/3,2/3]} ~\sigma^2_{[0,a]}(1-x) & \textrm{if $\frac{1}{2} < x \leq 1$} \end{array} \right.
$$
Together with~\eqref{discrv},~\eqref{sigma1} and~\eqref{sigma2} this proves the theorem.
\end{proof}

\begin{proof}[Proof of Theorem~\ref{th2}:~]
Suppose that $(n_k)_{k \geq 1}$ satisfies the Diophantine condition $\mathbf{D}$. Then by~\eqref{discrv} we have
\begin{equation} \label{lambdalo}
\Lambda^*(x) = \max_{a \in [0,1]} \sigma_{[0,a]} (x) \geq \sigma_{[0,1/2]} (x)
\end{equation}
for almost all $x$, where $\sigma_{[0,1/2]}(x)$ is defined according to~\eqref{sigma} for the function
$$
f(x) = \mathbf{I}_{[0,1/2]}(x).
$$ 
The Fourier series of $f$ is given by
$$
f(x) \sim \sum_{j=1}^\infty  \underbrace{\frac{\sin \pi j}{\pi j}}_{=:~a_j} \cos 2 \pi j x  + \underbrace{\left( \frac{1-\cos \pi j}{\pi j} \right)}_{=:~b_j} \sin {2 \pi j x}.
$$
Obviously $a_j =0$ for all $j$, while $b_j=2/(\pi j)$ for odd $j$ and $b_j=0$ for even $j$. Thus for this function $f$ the formula~\eqref{sigma} can be reduced to
$$
\sigma_f^2 (x) = \|f\|^2 + \sum_{\nu = -\infty}^\infty \sum_{j_1,j_2=1}^\infty \frac{\gamma_{j_1,j_2,\nu}}{2} b_{j_1} b_{j_2} \cos 2 \pi \nu x.
$$
Since all coefficients $b_j$ are non-negative, clearly also all products $b_{j_1} b_{j_2}$ must be non-negative, as are the numbers $\gamma_{j_1,j_2,\nu}$. Thus we have
$$
\int_0^1 \sigma_f^2 (x) ~dx = \|f\|^2 + \sum_{j_1,j_2=1}^\infty \frac{\gamma_{j_1,j_2,0}}{2} b_{j_1} b_{j_2} \geq \|f\|^2 = \frac{1}{4}.
$$
Together with~\eqref{lambdalo} this yields
$$
\|\Lambda^*\| \geq \frac{1}{2},
$$
which proves the first part of the theorem.\\

To prove the second part, we will show that for any fixed $x \in [0,1]$ we have
\begin{equation} \label{formu}
\int_0^{1/2} \sigma_{[a,a+1/2]}^2(x) ~da = \frac{1}{8}.
\end{equation}
This implies that for any $x$ there exists an interval $[a,a+1/2]$ such that 
$$
\sigma_{[a,a+1/2]}^2(x) \geq \frac{1}{4},
$$
which together with~\eqref{discrv} proves the second statement of Theorem~\ref{th2}.\\

To prove~\eqref{formu}, we note that for the Fourier coefficients of the function $\mathbf{I}_{[a,a+1/2]}(x)$ for some $a \in [0,1/2]$, given by
$$
\mathbf{I}_{[a,a+1/2]}(x) \sim \sum_{j=1}^\infty a_j \cos 2 \pi j x  + b_j \sin {2 \pi j x},
$$
we have
$$
a_j = a_j (a) = \frac{\sin 2 \pi j (a+1/2) - \sin 2\pi j a}{\pi j} = \left\{ \begin{array}{ll} -\frac{2 \sin 2 \pi j a}{j \pi} & \textrm{if $j$ is odd,}\\0 & \textrm{if $j$ is even} \end{array} \right.
$$
and
$$
b_j = b_j(a) = \frac{-\cos 2 \pi j (a+1/2) + \cos 2 \pi j a}{\pi j} = \left\{ \begin{array}{ll} \frac{2 \cos 2 \pi j a}{j \pi} & \textrm{if $j$ is odd,}\\0 & \textrm{if $j$ is even.} \end{array} \right.
$$
Note that the growth condition~\eqref{had} implies that $\gamma_{j_1,j_2,\nu}=0$ whenever $j_1 = j_2$. Thus according to~\eqref{sigma} we have
\begin{equation} \label{sigmasum}
\sigma_{[a,a+1/2]} (x) = \frac{1}{4} + \sum_{\nu = -\infty}^\infty \sum_{\substack{j_1 \neq j_2,\\j_1 \textrm{~and~}j_2 \textrm{~odd}}} \frac{\gamma_{j_1,j_2,\nu}}{2} \Big( \left(a_{j_1} a_{j_2} + b_{j_1} b_{j_2} \right) \cos 2 \pi \nu x + \left(b_{j_1}a_{j_2} - a_{j_1} b_{j_2}\right) \sin 2 \pi \nu x \Big)
\end{equation}
We have
\begin{eqnarray*}
& & \int_0^{1/2} \sum_{\nu = -\infty}^\infty \sum_{\substack{j_1 \neq j_2,\\j_1 \textrm{~and~}j_2 \textrm{~odd}}} \frac{\gamma_{j_1,j_2,\nu}}{2} a_{j_1} a_{j_2} \cos 2 \pi \nu x ~da \\
& = & \int_0^{1/2} \sum_{\nu = -\infty}^\infty \sum_{\substack{j_1 \neq j_2,\\j_1 \textrm{~and~}j_2 \textrm{~odd}}} 2 \gamma_{j_1,j_2,\nu}  \frac{\sin 2 \pi j_1 a}{j_1 \pi} \frac{\sin 2 \pi j_2 a}{j_2 \pi}  \cos 2 \pi \nu x ~da \\
& = & \int_0^{1/2} \sum_{\nu = -\infty}^\infty \sum_{\substack{j_1 \neq j_2,\\j_1 \textrm{~and~}j_2 \textrm{~odd}}} \gamma_{j_1,j_2,\nu}  \frac{\cos (2 \pi (j_1-j_2)a) - \cos (2 \pi (j_1+j_2)a)}{j_1 j_2 \pi^2} \cos 2 \pi \nu x ~da \\
& = & \sum_{\nu = -\infty}^\infty \sum_{\substack{j_1 \neq j_2,\\j_1 \textrm{~and~}j_2 \textrm{~odd}}}  \int_0^{1/2} \gamma_{j_1,j_2,\nu}  \frac{\cos (2 \pi (j_1-j_2)a) - \cos (2 \pi (j_1+j_2)a)}{j_1 j_2 \pi^2} \cos 2 \pi \nu x ~da \\
& = & 0.
\end{eqnarray*}
Here we used the fact that when $j_1$ and $j_2$ are both odd and $j_1 \neq j_2$, then both $j_1+j_2$ and $j_1-j_2$ are even and nonzero, and consequently all the integrals $\int_0^{1/2} \cos (2 \pi (j_1-j_2)a) ~da$ and $\int_0^{1/2} \cos (2 \pi (j_1+j_2)a) ~da$ are zero; furthermore, we used Lemma~\ref{lemmagamma} and the dominated convergence theorem to exchange the order of summation and integration. The other parts of the sum in~\eqref{sigmasum} can be treated in a similar way, and it turns out that their integrals are also equal to zero. Overall we get
$$
\int_0^{1/2} \sigma_{[a,a+1/2]} (x)~da =  \int_0^{1/2} \frac{1}{4}~da = \frac{1}{8},
$$
which is~\eqref{formu}. This proves the theorem.
\end{proof}

\begin{proof}[Proof of Theorem~\ref{th3}:~]
Theorem~\ref{th3} is a simple consequence of~\eqref{coslil} and Lemma~\ref{modkoksma} with $f(x)=\cos 2\pi x$.
\end{proof}

\begin{proof}[Proof of Theorem~\ref{th4}:~]
To prove the theorem, it is sufficient to show that for any distinct positive integers $m,n$ we have
$$
\int_0^1 \int_0^1 \mathbf{I}_{[a,a+z]} (m x) \mathbf{I}_{[a,a+z]} (n x) ~dx ~da= 0.
$$
For this purpose, we set
$$
g(a) = \int_0^1  \mathbf{I}_{[a,a+z]} (m x) \mathbf{I}_{[a,a+z]} (n x) ~dx.
$$
It is easily seen that the function $g$ has period 1, that is $g(a+1)=g(a)$. By noting that $\mathbf{I}_{[a,a+z]}(y) = \mathbf{I}_{[0,z]} (y-a)$ and by changing the variable $x$ to $x+a/m$ we have
\begin{eqnarray*}
g(a) & = & \int_0^1 \mathbf{I}_{[0,z]}(mx-a) \mathbf{I}_{[0,z]}(nx-a) ~dx \\
& = & \int_0^1 \mathbf{I}_{[0,z]}(mx) \mathbf{I}_{[0,z]}(nx+(n/m-1)a) ~dx
\end{eqnarray*}
(since the integrand has period 1, the integration interval does not change). Since $g$ has period 1, we have
\begin{eqnarray*}
\int_0^1 g(a) ~da & = & \frac{1}{m} \int_0^m g(a)~da\\
& = & \frac{1}{m} \int_0^1 \mathbf{I}_{[0,z]}(mx) \int_0^m \mathbf{I}_{[0,z]} (nx + (n/m-1) a)~da ~dx \\
& = & 0.
\end{eqnarray*}
This proves the theorem.
\end{proof}

%\bibliography{disc_extr}
%\bibliographystyle{abbrv}

\end{document}